\documentclass[reqno]{amsart}

\usepackage{amsfonts}
\usepackage{amssymb}
\usepackage{amsmath}
\usepackage{mathrsfs}
\usepackage{amsthm}
\usepackage[colorlinks,linkcolor=blue,citecolor=blue]{hyperref}

 \newtheorem{theorem}{Theorem}[section]
\newtheorem{corollary}[theorem]{Corollary}
\newtheorem{lemma}[theorem]{Lemma}
\newtheorem{proposition}[theorem]{Proposition}

\theoremstyle{remark}
\newtheorem{remark}[theorem]{Remark}

\theoremstyle{definition}
\newtheorem{definition}[theorem]{Definition}

\numberwithin{equation}{section}

\begin{document}

\title[Hessian type fully nonlinear elliptic equations]{On the exterior Dirichlet problem for Hessian type fully nonlinear elliptic equations}

\author[X.L. Li]{Xiaoliang Li}
\address[X.L. Li]{School of Mathematical Sciences\\
Beijing Normal University\\
	100875 Beijing\\
	P.R. China}
\email{xiaoliangli@mail.bnu.edu.cn}

\author[C. Wang]{Cong Wang}
\address[C. Wang]{School of Mathematical Sciences\\
Beijing Normal University\\
	100875 Beijing\\
	P.R. China}
\email{cwang@mail.bnu.edu.cn}

\thanks{}


\begin{abstract}
We treat the exterior Dirichlet problem for a class of fully nonlinear elliptic equations of the form $$f(\lambda(D^2u))=g(x),$$
with prescribed asymptotic behavior at infinity. The equations of this type had been studied extensively by Caffarelli--Nirenberg--Spruck \cite{Caffarelli1985}, Trudinger \cite{Trudinger1995} and many others, and there had been significant discussions on the solvability of the classical Dirichlet problem via the continuity method, under the assumption that $f$ is a concave function. In this paper, based on the Perron's method, we establish an exterior existence and uniqueness result for viscosity solutions of the equations, by assuming $f$ to satisfy certain structure conditions as in \cite{Caffarelli1985,Trudinger1995} but without requiring the concavity of $f$. The equations in our setting may embrace the well-known Monge--Amp\`ere equations, Hessian equations and Hessian quotient equations as special cases.
\end{abstract}

\maketitle

\section{Introduction}
\subsection{The setting of the equations}
Given a bounded domain $D$ in $\mathbb{R}^n$ with $n\geq 3$, we consider in this paper the Dirichlet problem for fully nonlinear, second-order partial differential equations of the form
\begin{equation}\label{eq:pro-eq}
f(\lambda(D^2u))=g(x)
\end{equation}
in the exterior domain $\mathbb{R}^n\setminus\overline{D}$, where $\lambda(D^2 u)=(\lambda_1,\cdots,\lambda_n)$ denotes the eigenvalue vector of the Hessian matrix $D^2u$, $f$ is a smooth symmetric function defined in an open convex symmetric cone $\Gamma\subset\mathbb{R}^n$, with vertex at the origin, such that
$$\Gamma^+:=\{\lambda\in\mathbb{R}^n: \lambda_i> 0, i=1,\cdots,n\}\subset\Gamma,$$
and $g$ is a positive function.

General equations of type \eqref{eq:pro-eq} were first treated by Caffarelli, Nirenberg and Spruck \cite{Caffarelli1985}, who proved the solvability of the classical Dirichlet problem
\begin{equation}\label{eq:Diri-interior}
\begin{cases}
f(\lambda(D^2u))=g(x) & \text{in }{D},\\
u=\varphi & \text{on }\partial D,
\end{cases}
\end{equation}
under various assumptions on the structure of the function $f$ as well as a geometric condition for $\partial D$. A typical example of $f$ embraced in \cite{Caffarelli1985} is $\sigma_k^{1/k}$ with $\Gamma=\Gamma_k$, where $\sigma_k$ is the $k$-th elementary symmetric function
\begin{equation}\label{eq:k-sigma}
\sigma_k(\lambda)=\sum_{1\leq i_1<\cdots<i_k\leq n}\lambda_{i_1}\cdots\lambda_{i_k},\quad k\in\{1,\cdots, n\}
\end{equation}
and $\Gamma_k$ is the Garding cone $$\Gamma_k=\{\lambda\in\mathbb{R}^n: \sigma_j(\lambda)>0\text{ for } j=1,\cdots, k\}.$$
In particular, when $k=n$, function \eqref{eq:k-sigma} corresponds to the famous Monge--Amp\`ere operator $$\sigma_n(\lambda(D^2u))=\det(D^2u).$$
Trudinger \cite{Trudinger1995} then extended the existence results in \cite{Caffarelli1985} to problem \eqref{eq:Diri-interior} with more general $f$ which allows the important examples of quotients of functions \eqref{eq:k-sigma}, the ones that were generally excluded in \cite{Caffarelli1985}, given by
\begin{equation}\label{eq:quotient}
\left(\frac{\sigma_k}{\sigma_l}\right)^{\frac{1}{k-l}},\quad 1\leq l<k\leq n.
\end{equation}

Specifically, the fundamental structure conditions on $f$ in \cite{Caffarelli1985,Trudinger1995} include
\begin{equation}\label{eq:increase-f}
\frac{\partial f}{\partial \lambda_i}>0 \quad \text{in }\Gamma, \quad i=1,\cdots,n,
\end{equation}
\begin{equation}\label{eq:concave-f}
f\text{ is a concave function in }\Gamma,
\end{equation}
and
\begin{equation}\label{eq:boundary-f}
\limsup_{\lambda\to\lambda_0} f(\lambda)<\inf_{\Omega} g\footnote{Here $\Omega$ is the domain where the function $g$ is defined.} \quad\text{for every }\lambda_0\in\partial\Gamma.
\end{equation}
In the study of fully nonlinear equations associated with form \eqref{eq:pro-eq}, conditions \eqref{eq:increase-f}--\eqref{eq:boundary-f} have become a standard setting for the function $f$ since the pioneer work \cite{Caffarelli1985}. Over the past few decades, besides \cite{Trudinger1995}, many significant contributions have been made to a priori estimates and the existence of solutions to problem \eqref{eq:Diri-interior} under conditions \eqref{eq:increase-f}--\eqref{eq:boundary-f} (possibly along with other more technical assumptions such as \eqref{eq:CNS-f} and \eqref{eq:T-f} below). These studies further extend the existence results in \cite{Caffarelli1985} from different directions. We refer to \cite{Guan1994,Guan1999,Guan2014,GJ2015,ITW2004,YYL1990,Urbas2002} and the references therein for the various extensions to degenerate problems, general domains, inhomogeneous terms $g=g(x,u,Du)$, and Riemannian manifolds.

However, as far as we know, the exterior counterpart of problem \eqref{eq:Diri-interior} with the general $f$ and $g$ has not been studied yet.

\subsection{The exterior Dirichlet problem}
The aim of this paper is to solve the exterior Dirichlet problem for equation \eqref{eq:pro-eq} in a general setting of the functions $f$ and $g$. Differently from interior problem \eqref{eq:Diri-interior}, in exterior domains our main concern is the existence and uniqueness of the solutions with prescribed asymptotic behavior at infinity. In this regard, the investigation for some special cases of $f$ including \eqref{eq:k-sigma}-\eqref{eq:quotient} has recently received increasing attention, which was motivated by Liouville-type results for the corresponding equations in unbounded domains.

As is well-known, a classical theorem due to J\"{o}rgens \cite{Jorgens1954}, Calabi \cite{Calabi1958} and Pogorelov \cite{Pogorelov1972} states that any convex entire solution of the Monge--Amp\`ere equation
\begin{equation}\label{eq:M-A}
\det(D^2u)=1
\end{equation}
(i.e. \eqref{eq:pro-eq} with $f=\sigma_n$ and $g\equiv1$) must be a quadratic polynomial; see also \cite{Caffarelli1995,Cheng-Yau-1986,Jost2001}. Caffarelli and Li \cite{Caffarelli-Li-2003} then extended this rigidity result to the setting of exterior domains. They proved that if $u$ is a convex solution of \eqref{eq:M-A} outside a bounded convex domain of $\mathbb{R}^n$ ($n\geq3$), then there exist a $n\times n$ symmetric positive definite matrix $A$ with $\det (A)=1$, a vector $b\in\mathbb{R}^n$ and a constant $c\in\mathbb{R}$ such that
\begin{equation}
\label{eq:C-Li}
\lim_{|x|\to\infty}\left|u(x)-\left(\frac12x^TAx+b\cdot x+c\right)\right|=0
\end{equation}
with asymptotic order $|x|^{2-n}$.
This result was obtained in \cite{Caffarelli-Li-2003} by showing that asymptotics \eqref{eq:C-Li} actually holds for convex entire solutions of
\begin{equation}\label{eq:M-A-g}
\det(D^2u)=g(x),
\end{equation}
where $g\in C^0(\mathbb{R}^n)$ satisfies $\inf_{\mathbb{R}^n}g>0$ and
\begin{equation}\label{eq:C-Li-g}
\mathrm{support}\,(g-1)\text{ is bounded}.
\end{equation}
Subsequently, Bao, Li and Zhang \cite{Bao-Li-Zhang-2015} derived \eqref{eq:C-Li} for equation \eqref{eq:M-A-g} under a weaker condition than \eqref{eq:C-Li-g}, which is given by
\begin{equation}\label{eq:g}
\limsup_{|x|\to\infty}|x|^\beta|g(x)-1|<\infty
\end{equation}
for some constant $\beta>2$. More generally, the above Liouville properties were also exploited for certain $k$-Hessian equations and Hessian quotient equations (corresponding to \eqref{eq:pro-eq} where $f$ takes \eqref{eq:k-sigma} and \eqref{eq:quotient} respectively) both on the whole space and on exterior domains; see \cite{Bao2003,CY2010,Li-Li-Yuan-2019,LRW2016,SY2021,WB2022,Warren-Yuan2008,Yuan2002} and the references therein. In particular, the authors of \cite{Li-Li-Yuan-2019,Warren-Yuan2008,Yuan2002} mainly studied Liouville-type results for the special Lagrangian equations, corresponding to \eqref{eq:pro-eq} where $f$ is
\begin{equation}\label{eq:Lag}
\frac{1}{\Theta}\sum_{i=1}^n\arctan\lambda_i
\end{equation}
and $g\equiv1$, where $|\Theta|\in (0,\frac{n}{2}\pi)$. Whenever $\Theta\geq\frac{n-1}{2}\pi$, function \eqref{eq:Lag} defined in $\Gamma_n$ is another example of $f$ fulfilling \eqref{eq:increase-f}--\eqref{eq:boundary-f}.


Accordingly, ones were naturally led to consider whether the exterior Dirichlet problem for these special equations is well-posed when assigning a quadratic polynomial as the specifying condition at infinity. Indeed, by Perron's method, in \cite{Caffarelli-Li-2003,Bao-Li-Zhang-2015} the authors also established existence and uniqueness theorems for exterior solutions to Monge--Amp\`ere equations \eqref{eq:M-A} and \eqref{eq:M-A-g} with \eqref{eq:g}, in terms of prescribed boundary data and asymptotic behavior \eqref{eq:C-Li}. Their results were later improved by Li and Lu \cite{Li-Lu-2018}, who gave the sharp conditions for the solvability of the problems considered in \cite{Bao-Li-Zhang-2015,Caffarelli-Li-2003}. In the same spirit, for the constant right-hand side, the exterior Dirichlet problem for $k$-Hessian equations, Hessian quotient equations and special Lagrangian equations has also been studied in \cite{Bao-Li-Li-2014,Li-Li-2018,Li2019} in the viscosity sense, under a prescribed quadratic condition at infinity. Moreover, the extension of these studies to a general right-hand side $g$ satisfying \eqref{eq:g} was treated in \cite{Cao-Bao-2017,Jiang-Li-Li-2021-b} recently.


Concerning such an investigation for equation \eqref{eq:pro-eq} with general $f$, the first effort is made by Li and Bao \cite{Li-Bao-2014}. Under structural conditions \eqref{eq:increase-f} and \eqref{eq:boundary-f} and the assumption that there is a positive number $a^*$ such that $$f(a^*,\cdots,a^*)=1,$$ they addressed the existence and uniqueness of viscosity solutions to the exterior Dirichlet problem for \eqref{eq:pro-eq} in the case $g\equiv 1$, with prescribed asymptotic behavior \eqref{eq:C-Li} in which the matrix $A$ is restricted to be $a^*I$.\footnote{$I$ is the identity matrix.} Recently, Jiang, Li and Li \cite{Jiang-Li-Li-2021} extended the result in \cite{Li-Bao-2014} to general prescribed quadratic asymptotics, where one can assign more matrices $A$ in \eqref{eq:C-Li} (more precisely, $A\in \mathscr{A}$; see \eqref{eq:A-alpha} below).

\subsection{The main result}
In the present paper, by applying an adapted Perron's method, we shall generalize the existence result in \cite{Jiang-Li-Li-2021} to equation \eqref{eq:pro-eq} with the general right-hand side $g$, under assumption \eqref{eq:g}. For the framework of the function $f$, we are able to skip the requirement of concavity condition \eqref{eq:concave-f} as in \cite{Jiang-Li-Li-2021,Li-Bao-2014}, since here the construction of the solution does not rely on deriving a priori second-order estimates and using the Evans--Krylov theorem where \eqref{eq:concave-f} plays a crucial role (see for instance \cite{Caffarelli1985,Guan2014,GJ2015,Trudinger1995}). We would also like to point out that this generalization is not trivial. Indeed, as an essential tool here we are exploiting, the comparison principle for viscosity solutions to \eqref{eq:pro-eq} whenever $g\not\equiv1$ is not a straightforward adaption of those established when $g\equiv 1$, since the Aleksandrov maximum principle is usually not helpful; see the discussions in \cite{Trudinger1990,Jiang-Li-Li-2021-b}. Moreover, due to the abstract form of $f$ and the variance of $g$, it is a delicate issue to seek appropriate subsolutions and supersolutions of \eqref{eq:pro-eq} for carrying out the Perron process. Especially, we need to present a new technique for the construction of supersolutions in the more general setting \eqref{eq:pro-eq}, since we could neither directly pick quadratic polynomials as the desired supersolutions as adopted in \cite{Bao-Li-Li-2014,Caffarelli-Li-2003,Jiang-Li-Li-2021,Li-Li-2018,Li-Bao-2014,Li2019} for the case $g\equiv1$, nor merely try to obtain such ones in a way parallel to seeking subsolutions as handled in \cite{Bao-Li-Zhang-2015,Cao-Bao-2017,Jiang-Li-Li-2021-b} for those special $f$ from \eqref{eq:k-sigma} and \eqref{eq:quotient}; see Remark \ref{pf-rk:proof} for a detailed explanation.

In order to overcome the above difficulties, besides \eqref{eq:increase-f} and \eqref{eq:boundary-f}, we further assume that $f$ satisfies
\begin{gather}
\sum_{i=1}^n\lambda_i\frac{\partial f}{\partial \lambda_i}\geq\nu(f)\quad\text{in }\Gamma,\label{eq:nu-f}\\
\frac{\partial f}{\partial\lambda_{i_0}}=\max_{1\leq i\leq n}\frac{\partial f}{\partial\lambda_i}\quad\text{in }\Gamma,\label{eq:max-partial-f}
\end{gather}
and for each $\lambda\in\Gamma^+$ there is a number $R$ such that
\begin{equation}\label{eq:sigma-f}
f(\lambda_1,\lambda_2,\cdots,\lambda_n+R)\geq 1.
\end{equation}
Here $\nu:\mathbb{R}\to\mathbb{R}^+$ is a positive increasing function and $i_0\in\{1,\cdots,n\}$ is such that $\lambda_{i_0}=\min_{1\leq i\leq n}\lambda_i$.
The example \eqref{eq:k-sigma} meets the conditions \eqref{eq:nu-f}--\eqref{eq:sigma-f}; actually, conditions \eqref{eq:nu-f}--\eqref{eq:sigma-f} are the variants of those conditions assumed in, for instance, Caffarelli--Nirenberg--Spruck \cite{Caffarelli1985} and Trudinger \cite{Trudinger1990,Trudinger1995} (see Remark \ref{rk:type-extra} below for details). We also point out that conditions \eqref{eq:nu-f}--\eqref{eq:sigma-f} are possible to be removed in our result when $g$ satisfies extra restrictions,
in which case the function $f$ we are considering may also include examples \eqref{eq:quotient} and \eqref{eq:Lag} (see Remark \ref{rk:remove-extra} below for details). Therefore, our result is indeed a generalization of those presented in \cite{Bao-Li-Li-2014,Bao-Li-Zhang-2015,Caffarelli-Li-2003,Cao-Bao-2017,Jiang-Li-Li-2021-b,Li-Li-2018,Li2019}.

To state precisely our main result, we introduce some definitions and notations. First, under hypotheses \eqref{eq:increase-f} and \eqref{eq:boundary-f}, we recall the definition of the viscosity solution to equation \eqref{eq:pro-eq} following \cite{Caffarelli-Cabre-1995,Ishii1992,Urbas1990}.
\begin{definition}\label{def:visc}
Given an open set $\Omega\subset\mathbb{R}^n$, a function $u\in\mathrm{USC}(\Omega)$ $(\mathrm{LSC}(\Omega))$\footnote{$\mathrm{USC}(\Omega)$ and $\mathrm{LSC}(\Omega)$ respectively denote the set of upper and lower semicontinuous real valued functions on $\Omega$.} is said to be a viscosity subsolution (supersolution) of \eqref{eq:pro-eq} in $\Omega$ (or say that $u$ satisfies $f(\lambda(D^2u))\geq(\leq)\,g$ in $\Omega$ in the viscosity sense), if for any admissible function\footnote{A function $\psi$ being of class $C^2$ is called admissible if $\lambda(D^2\psi)\in\Gamma$.} $\psi\in C^2(\Omega)$ and any local maximum (minimum) $x_0$ of $u-\psi$, we have $$f(\lambda(D^2\psi(x_0)))\geq(\leq)\,g(x_0).$$
A function $u\in C^0(\Omega)$ is said to be a viscosity solution of \eqref{eq:pro-eq}, if it is both a viscosity subsolution and a viscosity supersolution of \eqref{eq:pro-eq}.
\end{definition}

In the rest of this paper, we always denote by $\lambda(A)=(\lambda_1,\lambda_2,\cdots,\lambda_n)$ the eigenvalue vector of a real $n\times n$ symmetric matrix $A$ with the ascending order, namely, $\lambda_1\leq\lambda_2\leq\cdots\leq\lambda_n$. We define
$$\mathcal{A}=\{A:A\text{ is a real } n\times n\text{ symmetric positive definite matrix with }f(\lambda(A))=1\}$$
and
\begin{equation}\label{eq:A-alpha}
\mathscr{A}=\{A:A\in\mathcal{A}\text{ with } \alpha(A)>1\},
\end{equation}
where
\begin{equation}\label{eq:alpha}
\alpha(A):=\frac{\lambda(A)\cdot\nabla f(\lambda(A))}{2\lambda_n(A)\frac{\partial f}{\partial\lambda_1}(\lambda(A))}.
\end{equation}
The main result of this paper is the following.
\begin{theorem}\label{thm:main}
Let $D$ be a smooth, bounded, strictly convex domain in $\mathbb{R}^n$ with $n\geq 3$ and let $\varphi\in C^2(\partial D)$. Let $f$ be as in \eqref{eq:pro-eq} and satisfy \eqref{eq:increase-f}, \eqref{eq:boundary-f} and \eqref{eq:nu-f}--\eqref{eq:sigma-f}. Let $g\in C^0(\mathbb{R}^n\setminus D)$ satisfy \eqref{eq:g} and $\inf_{\mathbb{R}^n\setminus D}g>0$.

 For any $A\in\mathscr{A}$ and $b\in\mathbb{R}^n$, there exists a constant $c_*$ depending only on $n$, $D$, $f$, $g$, $A$, $b$ and $\|\varphi\|_{C^2(\partial D)}$, such that for every $c>c_*$ there exists a unique viscosity solution $u\in C^0(\mathbb{R}^n\setminus D)$ to the problem
\begin{equation}\label{eq:pro}
\begin{cases}
f(\lambda(D^2u))=g(x)\quad\text{in } \mathbb{R}^n\setminus\overline{D},\\
u=\varphi\quad\text{on }\partial D,\\
\lim_{|x|\to\infty}\left|u(x)-\left(\frac12x^TAx+b\cdot x+c\right)\right|=0.
\end{cases}
\end{equation}
\end{theorem}

In order to compare Theorem \ref{thm:main} with related results available in literature, let us make some remarks on the assumption of the function $f$ we are considering.

\begin{remark}\label{rk:remove-extra}
Here conditions \eqref{eq:increase-f} and \eqref{eq:boundary-f} are fundamental, which ensure equation \eqref{eq:pro-eq} to be elliptic on admissible functions and its viscosity solutions of class $C^2$ to be admissible, respectively. However, \eqref{eq:nu-f}--\eqref{eq:sigma-f} are more technical and may be removed. Indeed, condition \eqref{eq:nu-f} is required in the comparison principle (see Lemma \ref{pf-thm:comparison}) and could be removed if $g\equiv 1$ (see \cite[Remark A.5]{Jiang-Li-Li-2021-b}). Conditions \eqref{eq:max-partial-f} and \eqref{eq:sigma-f} are used to construct a family of subsolutions and supersolutions to \eqref{eq:pro-eq} in exterior domains, both with quadratic asymptotics at infinity. When $g\geq1$ in $\mathbb{R}^n\setminus\overline{D}$, \eqref{eq:sigma-f} can be removed since in this case we are able to directly let quadratic polynomials $\frac12x^TAx+b\cdot x+c$ be the supersolutions needed in the proof. Based on this, when $g\equiv 1$ particularly, Theorem \ref{thm:main} was proved in \cite{Jiang-Li-Li-2021}. Concerning \eqref{eq:max-partial-f}, it would be removed when we are restricted to the case $A=a^*I$ in \eqref{eq:pro}, where one is allowed to utilize radial functions to seek desired subsolutions as treated in \cite{Li-Bao-2014} (see Remark \ref{sub-rk:radial-f} below for details).
\end{remark}

\begin{remark}\label{rk:type-extra}
Related to the type of conditions \eqref{eq:nu-f}--\eqref{eq:sigma-f}, we make a connection with those required in the study of interior problem \eqref{eq:Diri-interior}. Recall that Caffarelli--Nirenberg--Spruck \cite{Caffarelli1985} (see also \cite{Guan1994,YYL1990}) assumed the following (in addition to \eqref{eq:increase-f}--\eqref{eq:boundary-f}): for every $C>0$ and compact set $K$ in $\Gamma$, there is a number $R=R(C,K)$ such that
\begin{gather}
f(\lambda_1,\cdots,\lambda_{n-1},\lambda_n+R)\geq C\quad\text{for all }\lambda\in K,\label{eq:CNS-f}\\
f(R\lambda)\geq C\quad\text{for all }\lambda\in K.\label{eq:T-f}
\end{gather}
Condition \eqref{eq:nu-f}\footnote{By homogeneity, examples \eqref{eq:k-sigma} and \eqref{eq:quotient} clearly fulfill \eqref{eq:nu-f}, where the function $\nu$ corresponds to numbers $k$ and $1$, respectively, but by which example \eqref{eq:Lag} is excluded.} implies \eqref{eq:T-f}. We mention that \eqref{eq:nu-f} was first introduced by Trudinger \cite{Trudinger1990} in order to treat the Dirichlet problem for the prescribed curvature equations. Regarding \eqref{eq:max-partial-f}, as is well-known, it can be derived from concavity condition \eqref{eq:concave-f}.\footnote{See for instance \cite[Lemma 2.2]{Andrews1994}, from which examples \eqref{eq:k-sigma}, \eqref{eq:quotient} and \eqref{eq:Lag} satisfy \eqref{eq:max-partial-f}.} In addition, \eqref{eq:CNS-f} implies \eqref{eq:sigma-f}.\footnote{Condition \eqref{eq:sigma-f} is clearly satisfied by examples \eqref{eq:k-sigma} and \eqref{eq:Lag} but, in general, excludes example \eqref{eq:quotient}; the same thus holds for \eqref{eq:CNS-f} as well.}
\end{remark}

\begin{remark}\label{rk:A}
The set $\mathscr{A}$ would not be empty in Theorem \ref{thm:main}. Indeed, via conditions \eqref{eq:increase-f} \eqref{eq:boundary-f} and \eqref{eq:nu-f}, for each $\lambda\in\Gamma$, the function $f(t\lambda)$ varies monotonically from $r_0<1$ to $+\infty$ as $t$ goes from $0$ to $+\infty$. So there is $t_1=t_1(\lambda)>0$ such that $f(t_1\lambda)=1$. In particular,  for $\lambda=(1,\cdots,1)$, $t_1=a^*$, as already mentioned. Clearly, $a^*I\in\mathscr{A}$ since $\alpha(a^*I)=\frac{n}{2}>1$. Then, by continuity, $\mathscr{A}$ contains a neighborhood of $a^*I$ in $\mathcal{A}$.
\end{remark}

Now we comment the proof of Theorem \ref{thm:main}. It is based on an adapted Perron's method (Lemma \ref{pf-thm:Perron}). In the spirit of \cite{Bao-Li-Li-2014,Jiang-Li-Li-2021,Li-Li-2018,Li2019}, the crucial point consists in utilizing the so-called generalized symmetric functions specified below to seek out a family of appropriate subsolutions and supersolutions of \eqref{eq:pro-eq}, both with uniformly quadratic asymptotics at infinity. However, the strategy in \cite{Bao-Li-Li-2014,Li-Li-2018,Li2019} heavily relies on the explicit formula of the $k$-Hessian operators acting on generalized symmetric functions, which is unavailable in our case. Instead, here we first adapt the idea in \cite{Jiang-Li-Li-2021} to construct a family of admissible subsolutions $u_{\xi_1,\xi_2}$ and supersolutions $U_{\eta_1,\eta_2,\delta}$ near infinity by solving two second-order implicit ODEs. Then in order to deal with the issue near the boundary, we turn to prepare another subsolution $\underline{w}$, given by the supremum of barrier functions over the boundary points, and also another a family of fine supersolutions $v_{\zeta_1,\zeta_2}$. Finally, by adjusting delicately the parameters $\xi_i$ ($\eta_i$, $\delta$ and $\zeta_i$) with $i=1,2$, we splice $u_{\xi_1,\xi_2}$ ($U_{\eta_1,\eta_2,\delta}$) and $\underline{w}$ ($v_{\zeta_1,\zeta_2}$) well to obtain the desired subsolutions (supersolutions).

Throughout the paper, following \cite{Bao-Li-Li-2014}, we call $u$ a \emph{generalized symmetric function} with respect to a $n\times n$ symmetric matrix $A$ if it is a function of $s=\frac12x^TAx$, $x\in\mathbb{R}^n$, that is $u(x)=u(\frac12x^TAx)$. If $u$ is a subsolution (supersolution) of \eqref{eq:pro-eq} and is also a generalized symmetric function, we say that $u$ is a generalized symmetric subsolution (supersolution) of \eqref{eq:pro-eq}.

\subsection{Organization of the paper}
In Section \ref{sec:sub} and Subsection \ref{sec:sup-1}, we construct a family of generalized symmetric subsolutions and supersolutions of \eqref{eq:pro-eq}, respectively. Both of them are admissible and asymptotically quadratic near infinity. In Subsection \ref{sec:sup-2}, we also construct a family of radial supersolutions of \eqref{eq:pro-eq} with fine properties so that they may be spliced with the previously constructed supersolutions. Section \ref{sec:proof} is devoted to the proof of Theorem \ref{thm:main}.

\section{Generalized symmetric subsolutions}\label{sec:sub}
In this section, we shall work with generalized symmetric functions to seek subsolutions of equation \eqref{eq:pro-eq} with $g\in C^0(\mathbb{R}^n\setminus D)$ satisfying \eqref{eq:g}. In the spirit of \cite{Jiang-Li-Li-2021}, we first compare the values of $f$ at the eigenvalue vectors of generalized symmetric functions and at certain points in the cone $\Gamma$. Then by solving a second-order implicit ODE, we construct a family of admissible subsolutions of \eqref{eq:pro-eq} with uniformly quadratic asymptotics at infinity; see Proposition \ref{sub-pro:ode} and Corollary \ref{sub-pro:subsolution} below.

Throughout the section, we let
\begin{equation}\label{sub-eq:diag-A}
A =\mathrm{diag}(a_1, a_2,\cdots, a_n)\in\mathscr{A}\text{ with }a_1\leq a_2\leq\cdots\leq a_n
\end{equation}
and let $u=u(s)$ be a generalized symmetric function with respect to $A$ and of class $C^2$, where $s=\frac12x^TAx=\frac12\sum_{i=1}^na_ix_i^2$, $x\in\mathbb{R}^n$. For simplicity, we denote $a=\lambda(A)$.

We are trying to estimate $f(\lambda(D^2u))$ from below; see Lemma \ref{sub-lem:lower-f}. For this purpose, let us start with the estimate of $\lambda(D^2u)$. Since
\begin{equation}\label{sub-eq:D2u}
\frac{\partial^2u}{\partial x_i\partial x_j}=a_i\delta_{ij}u'+a_ia_jx_ix_ju''\end{equation}
where $u':=\frac{du}{ds}$ and $u'':=\frac{d^2u}{ds^2}$, one can easily see that if $a_1=a_2=\cdots=a_n=a^*$, then the eigenvalues of $D^2u$ are
\begin{equation}\label{sub-eq:D2u-a}
a^*u'+2a^*su'',a^*u', \cdots,a^*u'.
\end{equation}
However, such a precise representation of $\lambda(D^2u)$ is not available for general $A$. Nevertheless, we can exploit the following inequality.
\begin{lemma}\label{sub-lem:D2u}
Assume $u'(s)>0$ and $u''(s)\leq 0$. Then
\begin{equation}\label{sub-eq:esti-D2u-1}
a_iu'(s)+\sum_{j=1}^na_j^2x_j^2u''(s)\leq \lambda_i(D^2u(x))\leq a_iu'(s),\quad\forall\, 1\leq i\leq n.
\end{equation}
\end{lemma}

\begin{proof}
Inequality (2.4) was obtained in \cite[Lemma 1]{Jiang-Li-Li-2021}. We present the proof here for completeness.

Actually, \eqref{sub-eq:esti-D2u-1} was showed in \cite{Jiang-Li-Li-2021} by virtue of the following Wely theorem (see for instance \cite[Theorem 4.3.1]{Horn1985}):
\begin{theorem}[Weyl]\label{sub-thm:Weyl}
Let $A_1$ and $A_2$ be real $n\times n$ symmetric matrices. Then for each $1\leq i\leq n$,
\begin{equation}\label{sub-eq:Wue}
\lambda_i(A_1+A_2)\leq\lambda_{i+j}(A_1)+\lambda_{n-j}(A_2), \quad\forall\, 0\leq j\leq n-i,
\end{equation}
and
\begin{equation}\label{sub-eq:Wle}
\lambda_i(A_1+A_2)\geq\lambda_{i-j+1}(A_1)+\lambda_j(A_2),\quad\forall\, 1\leq j\leq i.
\end{equation}
\end{theorem}
To obtain \eqref{sub-eq:esti-D2u-1}, by \eqref{sub-eq:D2u} we write $D^2u=A_1+A_2$ where $A_1=u'A$ and $A_2$ is the symmetric matrix whose elements are $a_ia_jx_ix_ju''$.  Then $\lambda(A_1)=u'a$ and
$$\lambda(A_2)=\left(\sum_{j=1}^na_j^2x_j^2u'',0,\cdots,0\right).$$
Thus, \eqref{sub-eq:esti-D2u-1} follows by taking $j=0$ in \eqref{sub-eq:Wue} and $j=1$ in \eqref{sub-eq:Wle}.
\end{proof}

Lemma \ref{sub-lem:D2u} leads us to the following property.
\begin{lemma}\label{sub-lem:lower-f}
Let $f$ be as in \eqref{eq:pro-eq} and assume that \eqref{eq:increase-f} and \eqref{eq:max-partial-f} hold. If $u'(s)>0$ and $u''(s)\leq0$, then
\begin{equation}\label{sub-eq:lower-f}
f(\lambda(D^2u(x)))\geq f(a_1u'(s)+2a_nsu''(s),a_2u'(s),\cdots,a_nu'(s)),
\end{equation}
provided that $\lambda(D^2u(x))$ and
$(a_1u'(s)+2a_nsu''(s),a_2u'(s),\cdots,a_nu'(s))$ both belong to $\Gamma$.
\end{lemma}

\begin{proof}
By \eqref{sub-eq:esti-D2u-1}, we could write
\begin{equation}\label{sub-eq:lambda-theta}
\lambda_i(D^2u)=a_iu'+\theta_i\sum_{j=1}^na_j^2x_j^2u'',\quad\forall\,1\leq i\leq n,\end{equation}
where $\theta_i=\theta_i(s)\in(0,1)$ satisfying $\sum_{i=1}^n\theta_i=1$.
Set
$$\bar{a}=(a_1u'+\sum_{j=1}^na_j^2x_j^2u'',a_2u',\cdots,a_nu').$$
Since $u''\leq 0$, it is clear that $\bar{a}\in\Gamma$ when $(a_1u'+2a_nsu'',a_2u',\cdots,a_nu')\in\Gamma$.
By the mean value theorem, we have
\begin{equation}\label{sub-eq:mv-lambda-a}
f(\lambda(D^2u))-f(\bar{a})=(\theta_1-1)Uu''\frac{\partial f}{\partial\lambda_1}(\tilde{a})+Uu''\sum_{i=2}^n\theta_i\frac{\partial f}{\partial\lambda_i}(\tilde{a}),
\end{equation}
where $\tilde{a}$ is a point lying in the segment between $\bar{a}$ and $\lambda(D^2u)$, and $U=\sum_{j=1}^na_j^2x_j^2$. Notice from \eqref{eq:max-partial-f} that
\begin{equation*}
\frac{\partial f}{\partial\lambda_1}(\tilde{a})\geq\frac{\partial f}{\partial\lambda_i}(\tilde{a})>0,\quad\forall\,2\leq i\leq n.
\end{equation*}
Thus, we get from \eqref{sub-eq:mv-lambda-a} that
$$f(\lambda(D^2u))\geq f(\bar{a})\geq f(a_1u'+2a_nsu'',a_2u',\cdots,a_nu'),$$
where the second ``$\geq$" is due to assumption \eqref{eq:increase-f}.
\end{proof}

\begin{remark}\label{sub-rk:radial-f}
When $A=a^*I$, it is clear from \eqref{sub-eq:D2u-a} that the equality of \eqref{sub-eq:lower-f} holds without requiring assumption \eqref{eq:max-partial-f}.
\end{remark}

Now, we use generalized symmetric functions to give a control of $g$ in \eqref{eq:pro-eq}. In view of $\inf_{\mathbb{R}^n\setminus D}g>0$ and \eqref{eq:g}, there exist $C_0$ and $s_0>1$ such that for any $s\geq s_0$,
\begin{equation}\label{sub-eq:g}
0<\underline{g}(s):=1-C_0s^{-\frac{\beta}{2}}\leq g(x)\leq\bar{g}(s):=1+C_0s^{-\frac{\beta}{2}}.
\end{equation}

Let $w_0=w_0(s)$ be a positive decreasing function defined on $[s_0,+\infty)$, which is determined by
\begin{equation}\label{sub-eq:f-w0-g}
f(a_1w_0(s),\cdots,a_nw_0(s))=\bar{g}(s).
\end{equation}
This equality is validated by assumptions \eqref{eq:increase-f}, \eqref{eq:boundary-f} and \eqref{eq:nu-f}, as explained in Remark \ref{rk:A}. Moreover, the implicit function theorem implies that $w_0$ is smooth and there hold
\begin{equation}\label{sub-eq:dw0}
w_0(s)=1+O(s^{-\frac{\beta}{2}})\quad\text{and}\quad\frac{dw_0}{ds}=-\frac{C_0\beta s^{-\frac{\beta}{2}-1}}{2a\cdot\nabla f(aw_0)}=O(s^{-\frac{\beta}{2}-1})
\end{equation}
as $s\to+\infty$.

In order to make $u$ be an admissible subsolution of \eqref{eq:pro-eq}, i.e. $f(\lambda(D^2u))\geq g$, Lemma \ref{sub-lem:lower-f} together with \eqref{sub-eq:g} inspires us to consider the following second-order implicit ODE:
\begin{equation}\label{sub-eq:IODE}
f(a_1u'+2a_nsu'',a_2u',\cdots,a_nu')=\bar{g},\quad s>s_0.
\end{equation}
By performing qualitative analysis, we next study the global existence of solutions to \eqref{sub-eq:IODE} and determine their asymptotic behavior at infinity. More precisely, given the initial data
\begin{equation}\label{sub-eq:iode-ic}
u(s_0)=\xi_1\quad\text{and}\quad u'(s_0)=\xi_2,
\end{equation}
we will show

\begin{proposition}\label{sub-pro:ode}
Let $f$ be as in \eqref{eq:pro-eq} and assume that \eqref{eq:increase-f}, \eqref{eq:boundary-f}, \eqref{eq:nu-f} and \eqref{eq:max-partial-f} hold. Let $\xi_1\in\mathbb{R}$ and $\xi_2>w_0(s_0)$. Problem \eqref{sub-eq:IODE}-\eqref{sub-eq:iode-ic} admits a smooth solution $u_{\xi_1,\xi_2}$ defined on $[s_0,+\infty)$, such that $u_{\xi_1,\xi_2}'(s)>1$, $u_{\xi_1,\xi_2}''(s)<0$ for any $s>s_0$ and
\begin{equation*}
u_{\xi_1,\xi_2}(s)=s+\xi_1+\mu(s_0,\xi_2)+
\begin{cases}
O(s^{1-\min\{\alpha,\frac{\beta}{2}\}})&\text{if }\alpha\neq\frac{\beta}{2},\\
O(s^{1-\alpha}\ln s)&\text{if }\alpha=\frac{\beta}{2},
\end{cases}
\end{equation*}
as $s\to+\infty$, where $\alpha=\alpha(A)$ is defined by \eqref{eq:alpha}, $\beta>2$ is as in \eqref{sub-eq:g} and $\mu$ is a function on $[s_0,+\infty)\times(w_0(s_0),+\infty)$ given by \eqref{sub-eq:def-mu}.
\end{proposition}

To prove Proposition \ref{sub-pro:ode}, we need several lemmas.

\begin{lemma}\label{sub-lem:h}
Let $s\in [s_0,+\infty)$ and $w\in[w_0(s),+\infty)$. Assume that \eqref{eq:increase-f}, \eqref{eq:boundary-f} and \eqref{eq:nu-f} hold. Then there exists a function $h=h(s,w)$ such that
\begin{equation*}
(h(s,w),a_2w,\cdots,a_nw)\in\Gamma
\end{equation*}
and
\begin{equation}\label{sub-eq:f-h-g}
f(h(s,w),a_2w,\cdots,a_nw)=\bar{g}(s).
\end{equation}
Moreover, $h(s,w)$ is smooth and decreasing with respect to $s$ and $w$.
\end{lemma}

\begin{proof}
 Let us first point out the existence of $h$. Since $\Gamma$ is symmetric and convex, one has $\Gamma\subset\Gamma_1$. Thus, by \eqref{eq:boundary-f}, it holds that
$$f(\varepsilon,a_2w,\cdots,a_nw)\leq\inf_{\mathbb{R}^n\setminus D}g\leq \bar{g}$$
when $\varepsilon$ is small. On the other hand, by \eqref{eq:increase-f} we have $$f(a_1w,\cdots,a_nw)\geq f(a_1w_0,\cdots,a_nw_0)=\bar{g}.$$
Then the function $h$ exists via the mean value theorem. Also, the convexity of $\Gamma$ implies $$(h(s,w),a_2w,\cdots,a_nw)\in\Gamma.$$ Furthermore, the implicit function theorem yields that $h$ is smooth and
\begin{gather}
\frac{\partial h}{\partial s}=-\frac{C_0\beta s^{-\frac{\beta}{2}-1}}{2\frac{\partial f}{\partial\lambda_1}(h(s,w),a_2w,\cdots,a_nw)}<0,\notag\\
\frac{\partial h}{\partial w}=-\frac{\sum_{i=2}^na_i\frac{\partial f}{\partial\lambda_i}(h(s,w),a_2w,\cdots,a_nw)}{\frac{\partial f}{\partial\lambda_1}(h(s,w),a_2w,\cdots,a_nw)}<0.\label{sub-eq:partial-w-h}
\end{gather}
\end{proof}

In terms of Lemma \ref{sub-lem:h}, one can observe that if $u'\geq w_0$ for $s>s_0$ in equation \eqref{sub-eq:IODE}, then
\begin{equation*}
a_1u'+2a_nsu''=h(s,u').
\end{equation*}
Recalling \eqref{sub-eq:iode-ic}, this means that such $u'$ solves the initial problem
\begin{equation}\label{sub-eq:ODE-w}
\begin{cases}
\frac{dw}{ds}=\frac{h(s,w)-a_1w}{2a_ns},& s>s_0,\\
w(s_0)=\xi_2.
\end{cases}
\end{equation}
As a result, we may try to solve the first-order ODE \eqref{sub-eq:ODE-w} to acquire a solution of problem \eqref{sub-eq:IODE}-\eqref{sub-eq:iode-ic}. Indeed, by applying the Picard--Lindel\"of theorem we have
\begin{lemma}\label{sub-lem:PL-w}
Let $\xi_2>w_0(s_0)$. Problem \eqref{sub-eq:ODE-w} admits a unique smooth solution $w_{\xi_2}(s)$ defined on $[s_0,+\infty)$. Moreover, for any $s\in [s_0,+\infty)$, $w_{\xi_2}'(s)<0$, $w_{\xi_2}(s)>w_0(s)$ and $w_{\xi_2}(s)$ is strictly increasing with respect to $\xi_2$ such that $$\lim_{\xi_2\to+\infty}w_{\xi_2}(s)=+\infty.$$
\end{lemma}

\begin{proof}
Since $\frac{h(s,w)-a_1w}{2a_ns}\in C^\infty \left((s_0,+\infty)\times(w_0(s),+\infty)\right)$, the Picard--Lindel\"of theorem implies that problem \eqref{sub-eq:ODE-w} locally admits a unique solution $w_{\xi_2}$ such that $w_{\xi_2}>w_0$. We claim that $w_{\xi_2}$ can be extended to the whole interval $[s_0,+\infty)$. Note that $h(s,w_{\xi_2})<a_1w_{\xi_2}$, which implies that the local solution $w_{\xi_2}$ is decreasing and bounded. Thus, it suffices to show that $w_{\xi_2}$ does not touch $w_0$ at a finite $s_1\in(s_0,+\infty)$. Assume for contradiction that $w_{\xi_2}(s_1)=w_0(s_1)$ and $w_{\xi_2}>w_0$ in $(s_0,s_1)$. Then by \eqref{sub-eq:f-w0-g} and \eqref{sub-eq:f-h-g}, we get $$h(s,w_{\xi_2}(s_1))=h(s,w_0(s_1))=a_1w_0(s_1).$$ This means $w_{\xi_2}'(s_1^-)=0$ via \eqref{sub-eq:ODE-w}. But from \eqref{sub-eq:dw0},
$$\lim_{s\to s_1^{-}}\frac{w_{
\xi_2}(s_1)-w_{\xi_2}(s)}{s_1-s}\leq\lim_{s\to s_1^{-}}\frac{w_0(s_1)-w_0(s)}{s_1-s}=w_0'(s_1)<0.$$
Consequently, we deduce that $w_{\xi_2}$ exists globally and $w_{\xi_2}(s)>w_0(s)$ for $s\in [s_0,+\infty)$.

The proof of the assertion that $\lim_{\xi_2\to+\infty}w_{\xi_2}(s)=+\infty$ is similar to that of Lemma 5 in \cite{Jiang-Li-Li-2021}. We thus omit it.
\end{proof}

Furthermore, we derive the asymptotic behavior of solutions to problem \eqref{sub-eq:ODE-w}.
\begin{lemma}\label{sub-lem:asym-w}
Let $\alpha$ and $\beta$ be as in Proposition \ref{sub-pro:ode}. Let $w_{\xi_2}$ be the solution of \eqref{sub-eq:ODE-w} given in Lemma \ref{sub-lem:PL-w}. Then
\begin{equation}\label{sub-eq:asym-w-xi2}
w_{\xi_2}(s)-1=
\begin{cases}
O(s^{-\min\{\alpha,\frac{\beta}{2}\}})&\text{if }\alpha\neq\frac{\beta}{2},\\
O(s^{-\alpha}\ln s)&\text{if }\alpha=\frac{\beta}{2},
\end{cases}
\end{equation}
as $s\to+\infty$.
\end{lemma}

\begin{proof}
We first show
\begin{equation}\label{sub-eq:asym-w-1}
w_{\xi_2}(s)-w_0=O(s^{-\frac{a_1}{2a_n}})\quad\text{as } s\to+\infty.
\end{equation}

Since $w_{\xi_2}>w_0$ for $s\geq s_0$, we note that $h(s,w_{\xi_2})\leq h(s,w_0)=a_1w_0$, which yields
\begin{align}
\frac{dw_{\xi_2}}{ds}&=\frac{h(s,w_{\xi_2})-a_1w_0+a_1w_0-a_1w_{\xi_2}}{2a_ns}\label{sub-eq:pf-w-decom}\\
&\leq\frac{a_1w_0-a_1w_{\xi_2}}{2a_ns}.\notag
\end{align}
Thus, by Gronwall's inequality one obtains
\begin{align}
w_{\xi_2}(s)&\leq \left(e^{\int_{s_0}^s-\frac{a_1}{2a_nt}dt}\right)\left[\xi_2+\int_{s_0}^s\left(e^{\int_{s_0}^r\frac{a_1}{2a_nt}dt}\right)\frac{a_1w_0(r)}{2a_nr}dr\right]\notag\\
&=\xi_2s_0^{\frac{a_1}{2a_n}}s^{-\frac{a_1}{2a_n}}
+s^{-\frac{a_1}{2a_n}}\int_{s_0}^s\frac{a_1}{2a_n}r^{\frac{a_1}{2a_n}-1}w_0(r)\,dr\notag\\
&\leq\xi_2s_0^{\frac{a_1}{2a_n}}s^{-\frac{a_1}{2a_n}}+w_0(s)-s^{-\frac{a_1}{2a_n}}\int_{s_0}^sr^{\frac{a_1}{2a_n}}w_0'(r)\,dr.\notag
\end{align}
Via \eqref{sub-eq:dw0}, it is easy to see that \eqref{sub-eq:asym-w-1} holds.

Next let us improve asymptotic behavior \eqref{sub-eq:asym-w-1} to \eqref{sub-eq:asym-w-xi2}. Going back to \eqref{sub-eq:pf-w-decom}, we infer that
\begin{align*}
\frac{dw_{\xi_2}}{ds}
=&\frac{w_{\xi_2}-w_0}{2a_ns}\left[\frac{\partial h}{\partial w}(s,\theta w_{\xi_2}+(1-\theta)w_0)-a_1\right]\\
=&\frac{w_{\xi_2}-w_0}{2a_ns}\left[\frac{\partial h}{\partial w}(s,\theta w_{\xi_2}+(1-\theta)w_0)-\frac{\partial h}{\partial w}(s,w_0)+\frac{\partial h}{\partial w}(s,w_0)-a_1\right]\\
\leq&\frac{w_{\xi_2}-w_0}{2a_ns}\left[M_{\partial_wh}(w_{\xi_2}-w_0)+\left(\frac{\partial h}{\partial w}(s,w_0)-a_1\right)\right]\\
=:&\frac{w_{\xi_2}-w_0}{2a_ns}[I(s)+J(s)],
\end{align*}
where $M_{\partial_wh}$ is the modulus of continuity of $\frac{\partial h}{\partial w}$. Hence, using Gronwall's inequality again gives
\begin{align}
w_{\xi_2}(s)&\leq F(s)\left[\xi_2
-\int_{s_0}^s\frac{I(r)+J(r)}{2a_nr}F^{-1}(r)w_0(r)\,dr\right]\notag\\
&=\xi_2F(s)+F(s)
\left[F^{-1}(r)w_0(r)\Big|_{s_0}^s-\int_{s_0}^sF^{-1}(r)w_0'(r)\,dr\right]\notag\\
&\leq\xi_2F(s)+w_0(s)-F(s)\int_{s_0}^sF^{-1}(r)w_0'(r)\,dr\label{sub-eq:asym-w-2}
\end{align}
where $F(s)=e^{\int_{s_0}^s\frac{I(t)+J(t)}{2a_nt}dt}$.

 On the one hand, recalling \eqref{sub-eq:asym-w-1}, we see that
\begin{equation}\label{sub-eq:I-Mod}
I(s)\leq M_{\partial_wh}(Cs^{-\frac{a_1}{2a_n}})
\end{equation}
when $s$ is sufficiently large, where $C$ is some constant.
Since $\frac{\partial h}{\partial w}$ is Dini continuous, we have
\begin{equation*}
\int_{s}^{+\infty}\frac{M_{\partial_wh}(t^{-\frac{a_1}{2a_n}})}{2a_nt}dt=\frac{1}{a_1}\int_0^C\frac{M_{\partial_wh}(t)}{t}dt<+\infty
\end{equation*}
where $C=C(s)$ is a constant. Together with \eqref{sub-eq:I-Mod}, we deduce that
\begin{equation*}
\int_{s_0}^s\frac{I(t)}{2a_nt}\,dt<+\infty, \quad\text{as }s\to+\infty.
\end{equation*}
On the other hand, by virtue of \eqref{sub-eq:partial-w-h} and \eqref{sub-eq:dw0}, we infer that
$$J(s)=-\frac{a\cdot\nabla f(a_1w_0,a_2w_0,\cdots,a_nw_0)}{\frac{\partial f}{\partial\lambda_1}(a_1w_0,a_2w_0,\cdots,a_nw_0)}=-2a_n\alpha+O(s^{-\frac{\beta}{2}}), \quad s\to+\infty.$$
Thus, we get
\begin{equation}\label{sub-eq:asym-F}
F(s)=O(s^{-\alpha})\quad\text{as }s\to +\infty.
\end{equation}

Now, thanks to \eqref{sub-eq:asym-F} and \eqref{sub-eq:dw0}, it follows from \eqref{sub-eq:asym-w-2} that \eqref{sub-eq:asym-w-xi2} holds. This completes the proof.
\end{proof}

We are now ready to prove Proposition \ref{sub-pro:ode}.
\begin{proof}[Proof of Proposition \ref{sub-pro:ode}]
Let $w_{\xi_2}$ be given in Lemma \ref{sub-lem:PL-w}. Define
\begin{equation}\label{sub-eq:ode-u-xi}
u_{\xi_1,\xi_2}(s)=\xi_1+\int_{s_0}^sw_{\xi_2}(t)\,dt.
\end{equation}
Since $u_{\xi_1,\xi_2}'(s)=w_{\xi_2}(s)>w_0(s)$ for any $s>s_0$, via Lemma \ref{sub-lem:h} it is clear that $u_{\xi_1,\xi_2}$ is a smooth solution of problem \eqref{sub-eq:IODE}-\eqref{sub-eq:iode-ic}. Moreover, it follows from Lemma \ref{sub-lem:asym-w} that
\begin{equation*}
\begin{split}
u_{\xi_1,\xi_2}(s)&=\int_{s_0}^s(w_{\xi_2}(t)-1)\,dt+s-s_0+\xi_1\\
&=\int_{s_0}^{+\infty}(w_{\xi_2}(t)-1)\,dt+s-s_0+\xi_1-\int_{s}^{+\infty}(w_{\xi_2}(t)-1)\,dt\\
&=s+\xi_1+\mu(\xi_2)+
\begin{cases}
O(s^{1-\min\{\alpha,\frac{\beta}{2}\}})&\text{if }\alpha\neq\frac{\beta}{2},\\
O(s^{1-\alpha}\ln s)&\text{if }\alpha=\frac{\beta}{2},
\end{cases}
\end{split}
\end{equation*}
as $s\to+\infty$, where
\begin{equation}\label{sub-eq:def-mu}
\mu(s_0,\xi_2):=\int_{s_0}^{+\infty}(w_{\xi_2}(t)-1)\,dt-s_0<+\infty.
\end{equation}
This completes the proof.
\end{proof}

Finally, to obtain a subsolution of \eqref{eq:pro-eq}, by combining Proposition \ref{sub-pro:ode} with Lemma \ref{sub-lem:lower-f}, we conclude the following.

\begin{corollary}\label{sub-pro:subsolution}
The function $u_{\xi_1,\xi_2}$ given by Proposition \ref{sub-pro:ode} is an admissible subsolution of equation \eqref{eq:pro-eq} in the domain $\{x\in\mathbb{R}^n:\frac12\sum_{i=1}^na_ix_i^2>s_0\}$.
\end{corollary}

\begin{proof}
In view of Lemma \ref{sub-lem:lower-f} and formulas \eqref{sub-eq:g} and \eqref{sub-eq:IODE}, it suffices to show that $u_{\xi_1,\xi_2}$ is admissible, i.e. $\lambda(D^2u_{\xi_1,\xi_2})\in\Gamma$. If $\lambda(D^2u_{\xi_1,\xi_2}(x_0))\notin\Gamma$ for some $x_0\in\{x\in\mathbb{R}^n:\frac12\sum_{i=1}^na_ix_i^2>s_0\}$, then there exists $\epsilon_0\geq0$ such that
$$\lambda(D^2u_{\xi_1,\xi_2}(x_0)+\epsilon_0A)\in\partial\Gamma$$
and
$$\lambda(D^2u_{\xi_1,\xi_2}(x_0)+\epsilon A)\in\Gamma\quad\text{for any }\epsilon>\epsilon_0.$$
Note that for $s>s_0$,
$$(u_{\xi_1,\xi_2}+\epsilon s)'>0\quad\text{and}\quad(u_{\xi_1,\xi_2}+\epsilon s)''<0.$$
Hence, by applying Lemma \ref{sub-lem:lower-f} to the function $u_{\xi_1,\xi_2}+\epsilon s$ we obtain that, for any $\epsilon>\epsilon_0$,
\begin{equation*}
\begin{split}
&f(\lambda(D^2u_{\xi_1,\xi_2}+\epsilon A))\\=&\,f(\lambda(D^2(u_{\xi_1,\xi_2}+\epsilon s)))\\
\geq&\,f(a_1(u_{\xi_1,\xi_2}'+\epsilon)+2a_nsu_{\xi_1,\xi_2}'',a_2(u_{\xi_1,\xi_2}'+\epsilon),\cdots,a_n(u_{\xi_1,\xi_2}'+\epsilon))\\
\geq&\,f(a_1u_{\xi_1,\xi_2}'+2a_nsu_{\xi_1,\xi_2}'',a_2u_{\xi_1,\xi_2}',\cdots,a_nu_{\xi_1,\xi_2}')=\bar{g}(s).
\end{split}
\end{equation*}
This contradicts the condition \eqref{eq:boundary-f}.

Consequently, $u_{\xi_1,\xi_2}$ is admissible and by Lemma \ref{sub-lem:lower-f} it satisfies
\begin{align*}
f(\lambda(D^2u_{\xi_1,\xi_2}(x)))&\geq f\left(a_1u_{\xi_1,\xi_2}'+2a_nsu_{\xi_1,\xi_2}'',a_2u_{\xi_1,\xi_2}',\cdots,a_nu_{\xi_1,\xi_2}'\right)\\
&=\bar{g}(s)\geq g(x)
\end{align*}
for any $x\in\{x\in\mathbb{R}^n:\frac12\sum_{i=1}^na_ix_i^2>s_0\}$. The proof is completed.
\end{proof}

\section{Generalized symmetric supersolutions}\label{sec:sup}
This section is devoted to the construction of supersolutions of equation \eqref{eq:pro-eq} and contains two parts. We first seek out a family of generalized symmetric supersolutions with uniformly quadratic asymptotics at infinity, by adapting carefully the idea of seeking such subsolutions in Section \ref{sec:sub}; see Proposition \ref{sup-pro:ode} and Corollary \ref{sup-pro:supersolution} below. However, the supersolutions obtained in Corollary \ref{sup-pro:supersolution} only stand near infinity and might not work near $\partial D$. Therefore, we were further led to look for another a family of fine supersolutions of \eqref{eq:pro-eq} outside $D$ (see Proposition \ref{sup-pro:rad-ode}), with which one can splice the former to apply Perron's method to prove Theorem \ref{thm:main} in the next section.

We will assume throughout the section that $g\in C^0(\mathbb{R}^n\setminus D)$ satisfies \eqref{eq:g} and $0<\inf_{\mathbb{R}^n\setminus D}g<1$. Since if $\inf_{\mathbb{R}^n\setminus D}g\geq1$, quadratic polynomials $\frac12 x^TAx+b\cdot x+c$ with $A\in\mathscr{A}$ may serve as the supersolutions of \eqref{eq:pro-eq} outside $D$.

\subsection{Generalized symmetric supersolutions near infinity}\label{sec:sup-1}
In this subsection, we let $A\in\mathscr{A}$ be of form \eqref{sub-eq:diag-A} and let $u=u(s)\in C^2$ be a generalized symmetric function with respect to $A$, where $s=\frac12x^TAx=\frac12\sum_{i=1}^na_ix_i^2$, $x\in\mathbb{R}^n$. As in Section \ref{sec:sub}, we denote $a=\lambda(A)$.

Let us start by estimating $f(\lambda(D^2u))$ from above with the following inequality.
\begin{lemma}\label{sup-lem:D2u}
Assume $u'(s)>0$ and $u''(s)\geq0$. Then
\begin{equation*}
a_iu'(s)\leq\lambda_i(D^2u(x))\leq a_iu'(s)+\sum_{j=1}^na_j^2x_j^2u''(s),\quad\forall\,1\leq i\leq n.
\end{equation*}
\end{lemma}
\begin{proof}
The proof follows exactly as that of Lemma \ref{sub-lem:D2u}. Actually, the only difference is that $\lambda(A_2)$ there is now changed to $(0,\cdots,0,\sum_{j=1}^na_j^2x_j^2u'')$.
\end{proof}

\begin{lemma}\label{sup-lem:upper-f}
Let $\delta>0$ and let $f$ be as in \eqref{eq:pro-eq} with \eqref{eq:increase-f} and \eqref{eq:max-partial-f} holding. Assume that $u'>0$ and $u''\geq0$ in $(0,+\infty)$. If
\begin{equation*}
\lim_{s\to+\infty}u'(s)=1\quad\text{and}\quad\lim_{s\to+\infty}su''(s)=0,
\end{equation*}
then there exists $\bar{s}=\bar{s}(f,A,\delta,u',u'')>0$ such that for any $s>\bar{s}$,
$$f(\lambda(D^2u))\leq f(a_1u'+(2a_n+\delta)su'',a_2u',\cdots,a_nu').$$
\end{lemma}

\begin{proof}
From Lemma \ref{sup-lem:D2u}, it is clear that $\lambda(D^2u)\in\Gamma$ and we can write it as in \eqref{sub-eq:lambda-theta}. Set
$$\bar{a}_{\delta}=(a_1u'+(2a_n+\delta)su'',a_2u',\cdots,a_nu').$$
 It is seen that $\bar{a}_{\delta}\in\Gamma$ as well. Then similar to \eqref{sub-eq:mv-lambda-a}, we have
\begin{equation}\label{sup-eq:mv-lambda-a-del}
f(\lambda(D^2u))-f(\bar{a}_{\delta})=[\theta_1U-(2a_n+\delta)s]u''\frac{\partial f}{\partial\lambda_1}(\tilde{a}_{\delta})+Uu''\sum_{i=2}^n\theta_i\frac{\partial f}{\partial\lambda_i}(\tilde{a}_{\delta}),
\end{equation}
where $\tilde{a}_{\delta}$ is a point lying in the segment between $\lambda(D^2u)$ and $\bar{a}_{\delta}$, and $U=\sum_{j=1}^na_j^2x_j^2$.

Since
$$Uu''=O(su'')\quad\text{as }s\to+\infty,$$
one obtains
$$\tilde{a}_{\delta}\to a\quad\text{and}\quad\nabla f(\tilde{a}_{\delta})\to\nabla f(a)\quad\text{as }s\to+\infty.$$
That is, for small $\epsilon>0$, there exists $\bar{s}$ such that
$$|\nabla f(\tilde{a}_{\delta})-\nabla f(a)|<\epsilon\quad\text{for }s\geq\bar{s}.$$
Thus, by \eqref{sup-eq:mv-lambda-a-del} and \eqref{eq:max-partial-f} we get
\begin{align*}
&f(\lambda(D^2u))-f(\bar{a}_{\delta})\\
\leq&\,[U-(2a_n+\delta)s]u''\frac{\partial f}{\partial\lambda_1}(a)+[(1-2\theta_1)U+(2a_n+\delta)s]u''\epsilon\\
\leq&-\delta su''\frac{\partial f}{\partial\lambda_1}(a)+(4a_n+\delta)su''\epsilon.
\end{align*}
Via \eqref{eq:increase-f}, the assertion follows by letting $0<\epsilon<\frac{\delta}{4a_n+\delta}\frac{\partial f}{\partial\lambda_1}(a)$.
\end{proof}

Now let $s_0$ in \eqref{sub-eq:g} further satisfy
\begin{equation}\label{sup-eq:g-s0}
\inf_{\mathbb{R}^n\setminus D}g\leq\underline{g}(s)=1-C_0s^{-\frac{\beta}{2}}\leq g(x)\quad\text{for }s\geq s_0.
\end{equation}
In order to find supersolutions of \eqref{eq:pro-eq}, Lemma \ref{sup-lem:upper-f} suggests us to study the following second-order ODE
\begin{equation}\label{sup-eq:iode}
f(a_1u'+(2a_n+\delta)su'',a_2u',\cdots,a_nu')=\underline{g},\quad s>s_0,
\end{equation}
with the initial data
\begin{equation}\label{sup-eq:iode-ic}
u(s_0)=\eta_1\quad\text{and}\quad u'(s_0)=\eta_2.
\end{equation}
 We will show the global existence of solutions to problem \eqref{sup-eq:iode}-\eqref{sup-eq:iode-ic} and determine their asymptotic behavior at infinity, in a way parallel to how we exploited Lemmas \ref{sub-lem:h}--\ref{sub-lem:asym-w} to solve problem \eqref{sub-eq:IODE}-\eqref{sub-eq:iode-ic} in Section \ref{sec:sub}.
,

First, as an analogue of $w_0$ determined by \eqref{sub-eq:f-w0-g}, let $W_0=W_0(s)$ be the positive increasing function defined on $[s_0,+\infty)$ such that
\begin{equation}\label{sup-eq:f-W0-g}
f(a_1W_0(s),\cdots,a_nW_0(s))=\underline{g}(s).
\end{equation}

\begin{lemma}\label{sup-lem:H}
Let $s\in [s_0,+\infty)$ and $w\in(0,W_0(s)]$. Assume that \eqref{eq:increase-f}, \eqref{eq:boundary-f} and \eqref{eq:sigma-f} hold. Then there exists a positive smooth function $H=H(s,w)$ such that
\begin{equation*}
f(H(s,w),a_2w,\cdots,a_nw)=\underline{g}(s).
\end{equation*}
Moreover, $H(s,w)$ is increasing with respect to $s$ and decreasing with respect to $w$.
\end{lemma}

\begin{proof}
Thanks to \eqref{eq:boundary-f} and \eqref{eq:sigma-f}, it is clear that $H(s,w)$ exists. Then the smoothness and monotonicity of $H(s,w)$ can be easily obtained via \eqref{eq:increase-f} and the implicit function theorem.
\end{proof}

With Lemma \ref{sup-lem:H}, we may reduce the solvability of problem \eqref{sup-eq:iode}-\eqref{sup-eq:iode-ic} to that of the following initial problem
\begin{equation}\label{sup-eq:ode-W}
\begin{cases}
\frac{dw}{ds}=\frac{H(s,w)-a_1w}{(2a_n+\delta)s},&s>s_0,\\
w(s_0)=\eta_2.
\end{cases}
\end{equation}
Indeed, through the analysis performed previously for problem \eqref{sub-eq:ODE-w}, we have
\begin{lemma}\label{sup-lem:PL-W}
Let $\delta>0$ and $0<\eta_2<W_0(s_0)$. Problem \eqref{sup-eq:ode-W} admits a unique smooth solution $W_{\eta_2,\delta}(s)$ defined on $[s_0,+\infty)$. Moreover, $W_{\eta_2,\delta}'>0$, $W_{\eta_2,\delta}<W_0$ and
\begin{equation*}
W_{\eta_2,\delta}(s)-1=
\begin{cases}
O(s^{-\min\{\alpha_\delta,\frac{\beta}{2}\}})&\text{if }\alpha_\delta\neq\frac{\beta}{2},\\
O(s^{-\alpha_\delta}\ln s)&\text{if }\alpha_\delta=\frac{\beta}{2},
\end{cases}
\end{equation*}
as $s\to+\infty$, where $\beta$ is as in \eqref{sub-eq:g} and
\begin{equation}\label{sup-eq:def-alpha-d}
\alpha_{\delta}=\frac{a\cdot\nabla f(a)}{(2a_n+\delta)\frac{\partial f}{\partial\lambda_1}(a)}.
\end{equation}
\end{lemma}

\begin{proof}
First, as in the proof of Lemma \ref{sub-lem:PL-w}, the local existence and uniqueness of $W_{\eta_2,\delta}$ can be easily obtained by applying the Picard--Lindel\"of theorem. Then a contradiction argument shows that $W_{\eta_2,\delta}$ is strictly increasing without touching the function $W_0$ from below. The global existence of $W_{\eta_2,\delta}$ thus follows. Noting that $W_0$ also behaves as \eqref{sub-eq:dw0}, it is not difficult to derive the asymptotic property of $W_{\eta_2,\delta}$ as in the proof of Lemma \ref{sub-lem:asym-w}.

For the above reason, we omit the proof.
\end{proof}

Lemma \ref{sup-lem:PL-W} together with Lemma \ref{sup-lem:H} leads us to
\begin{proposition}\label{sup-pro:ode}
Let $f$ be as in \eqref{eq:pro-eq} and assume that \eqref{eq:increase-f}, \eqref{eq:boundary-f}, \eqref{eq:max-partial-f} and \eqref{eq:sigma-f} hold. Let $\delta>0$, $\eta_1\in\mathbb{R}$ and $\eta_2\in (0,W_0(s_0))$. Problem \eqref{sup-eq:iode}-\eqref{sup-eq:iode-ic} admits a smooth solution $U_{\eta_1,\eta_2,\delta}$ defined on $[s_0,+\infty)$, such that $U_{\eta_1,\eta_2,\delta}'>0$ and $U_{\eta_1,\eta_2,\delta}''>0$. Moreover, if $\alpha_\delta>1$, then
\begin{equation*}
U_{\eta_1,\eta_2,\delta}(s)=s+\eta_1+\bar\mu+
\begin{cases}
O(s^{1-\min\{\alpha_\delta,\frac{\beta}{2}\}})&\text{if }\alpha_\delta\neq\frac{\beta}{2},\\
O(s^{1-\alpha_\delta}\ln s)&\text{if }\alpha_\delta=\frac{\beta}{2},
\end{cases}
\end{equation*}
as $s\to+\infty$, where $\alpha_\delta$ is defined by \eqref{sup-eq:def-alpha-d}, $\beta>2$ is as in \eqref{sub-eq:g} and $\bar\mu$ depends on $\eta_2$ and $\delta$.
\end{proposition}

\begin{proof}
Let $W_{\eta_2,\delta}$ be given by Lemma \ref{sup-lem:PL-W}. Define
\begin{equation}\label{sup-eq:def-U-eta}
U_{\eta_1,\eta_2,\delta}(s)=\eta_1+\int_{s_0}^sW_{\eta_2,\delta}(t)\,dt.
\end{equation}
Since $0<U_{\eta_1,\eta_2,\delta}'(s)=W_{\eta_2,\delta}(s)<W_0(s)$ for any $s>s_0$, it follows from Lemma \ref{sup-lem:H} that $U_{\eta_1,\eta_2,\delta}$ is a smooth solution of problem \eqref{sup-eq:iode}-\eqref{sup-eq:iode-ic}. Moreover, when $\alpha_\delta>1$, we have
\begin{equation*}
\begin{split}
U_{\eta_1,\eta_2,\delta}(s)&=\eta_1+\int_{s_0}^s(W_{\eta_2,\delta}(t)-1)\,dt+s-s_0\\
&=s+\eta_1+\int_{s_0}^{+\infty}(W_{\eta_2,\delta}(t)-1)\,dt-s_0-\int_{s}^{+\infty}(W_{\eta_2,\delta}(t)-1)\,dt.
\end{split}
\end{equation*}
Thus, we complete the proof by setting
$$\bar\mu=\int_{s_0}^{+\infty}(W_{\eta_2,\delta}(t)-1)\,dt-s_0<+\infty.$$
\end{proof}

As an immediate consequence of combining Proposition \ref{sup-pro:ode} with Lemma \ref{sup-lem:upper-f}, we conclude the following.
\begin{corollary}\label{sup-pro:supersolution}
The function $U_{\eta_1,\eta_2,\delta}$ given by Proposition \ref{sup-pro:ode} is an admissible supersolution of equation \eqref{eq:pro-eq} in the domain $\{x\in\mathbb{R}^n:\frac12\sum_{i=1}^na_ix_i^2>\max\{\bar{s},s_0\}\}$. Here $\bar{s}$ is as in Lemma \ref{sup-lem:upper-f} and $s_0$ is as in \eqref{sup-eq:g-s0}.
\end{corollary}

\begin{remark}\label{sup-rk:bar-s}
Observe from Lemma \ref{sup-lem:upper-f} that $\bar{s}$ above depends on $\eta_2$ and $\delta$ but is independent of $\eta_1$. This fact will play an essential role in the proof of Theorem \ref{thm:main}.
\end{remark}

\subsection{Radial supersolutions outside $D$}\label{sec:sup-2}
 This subsection aims to find another a family of supersolutions of \eqref{eq:pro-eq} outside $D$ from certain radial functions, in order to serve the needs of proving Theorem \ref{thm:main} in Section \ref{sec:proof}. Without loss of generality, we assume that the unit ball of $\mathbb{R}^n$ is contained in $D$.

Let $\tilde{a}>0$ be such that
$$f(\tilde{a},\cdots,\tilde{a})=\inf_{\mathbb{R}^n\setminus D} g$$
and let $s=\frac12\tilde{a}|x|^2$, $x\in\mathbb{R}^n$. In Proposition \ref{sup-pro:rad-ode} below, we obtain an admissible function $v=v(s)$ solving the following problem
\begin{equation}\label{sup-eq:rad-ode}
\begin{cases}
f(\lambda(D^2v))=f(\tilde{a}v'+2\tilde{a}sv'',\tilde{a}v',\cdots,\tilde{a}v')=\inf_{\mathbb{R}^n\setminus D} g, & s>\frac12\tilde{a},\\
v(\frac12\tilde{a})=\zeta_1,\quad v'(\frac12\tilde{a})=\zeta_2.
\end{cases}
\end{equation}
Here $v':=\frac{dv}{ds}$ and $v'':=\frac{d^2v}{ds^2}$. Clearly, such $v$ is a supersolution of equation \eqref{eq:pro-eq} in $\mathbb{R}^n\setminus\overline{D}$.

\begin{proposition}\label{sup-pro:rad-ode}
Let $f$ be as in \eqref{eq:pro-eq} and assume that \eqref{eq:increase-f}, \eqref{eq:boundary-f} and \eqref{eq:sigma-f} hold. Let $\zeta_1\in\mathbb{R}$ and $\zeta_2>1$. Problem \eqref{sup-eq:rad-ode} admits an admissible solution $v_{\zeta_1,\zeta_2}$. Moreover, for any $s>\frac12\tilde{a}$, $v_{\zeta_1,\zeta_2}'(s)>1$ and $v_{\zeta_1,\zeta_2}''(s)<0$.
\end{proposition}

\begin{proof}
This result can be proved in the same way seeking a smooth solution of problem \eqref{sub-eq:IODE}-\eqref{sub-eq:iode-ic} as handled in Section \ref{sec:sub} (see Proposition \ref{sub-pro:ode}).

Indeed, as argued for Lemma \ref{sub-lem:h}, we first observe that for each $w>0$ there exists a function $\bar{h}=\bar{h}(w)$ such that
\begin{equation*}
(\bar{h}(w),\tilde{a}w,\cdots,\tilde{a}w)\in\Gamma
\end{equation*}
and
\begin{equation*}
f(\bar{h}(w),\tilde{a}w,\cdots,\tilde{a}w)=\inf_{\mathbb{R}^n\setminus D} g.
\end{equation*}
 In order to obtain a solution of \eqref{sup-eq:rad-ode}, this leads us to study
\begin{equation}\label{sup-eq:ode-bar-w}
\begin{cases}
\frac{dw}{ds}=\frac{\bar{h}(w)-\tilde{a}w}{2\tilde{a}s},& s>\frac12\tilde{a},\\
w(\frac12\tilde{a})=\zeta_2.
\end{cases}
\end{equation}
As argued for Lemma \ref{sub-lem:PL-w}, by applying the Picard--Lindel\"of theorem and the theorem of maximal interval of existence for the solution of the initial value problem of ODEs, we easily infer that problem \eqref{sup-eq:ode-bar-w} admits a unique smooth solution $\bar{w}_{\zeta_2}$ defined on $[\frac12\tilde{a},+\infty)$, such that for any $s>\frac12\tilde{a}$,
\begin{equation}\label{sup-eq:d-w-zeta}
\bar{w}_{\zeta_2}(s)>1\quad\text{and}\quad\bar{w}_{\zeta_2}'(s)<0.
\end{equation}
 In addition, $\bar{w}_{\zeta_2}$ is strictly increasing with respect to $\zeta_2$ such that
\begin{equation}\label{sup-eq:w-zeta-infty}
\lim_{\zeta_2\to+\infty}\bar{w}_{\zeta_2}=+\infty.
\end{equation}

 Let
\begin{equation}\label{sup-eq:def-v-zeta}
v_{\zeta_1,\zeta_2}(s)=\zeta_1+\int_{\frac12\tilde{a}}^{s}\bar{w}_{\zeta_2}(t)\,dt.
\end{equation}
Then $v_{\zeta_1,\zeta_2}$ is a smooth admissible solution of \eqref{sup-eq:rad-ode}.
\end{proof}

\section{Proof of Theorem \ref{thm:main}}\label{sec:proof}
In this section, we prove Theorem \ref{thm:main} by applying an adapted Perron's method (see Lemma \ref{pf-thm:Perron} below). The main ingredient of the proof is to demonstrate the existence of a viscosity subsolution $\underline{u}$ of \eqref{eq:pro-eq} with prescribed Dirichlet boundary value and asymptotic behavior at infinity, and also a viscosity supersolution $\bar{u}\geq\underline{u}$ but agreeing on $\underline{u}$ at infinity. As we will see later, such a subsolution could be obtained by splicing the generalized symmetric subsolution $u_{\xi_1,
\xi_2}$ in Corollary \ref{sub-pro:subsolution} and the supremum of barrier functions over the boundary points of $D$ (see Lemma \ref{pf-lem:w-xi} below). Analogously, such a supersolution is generally constructed by splicing the generalized symmetric supersolution $U_{\eta_1,
\eta_2,\delta}$ in Corollary \ref{sup-pro:supersolution} and the radial supersolution $v_{\zeta_1,\zeta_2}$ given by Proposition \ref{sup-pro:rad-ode}. However, the above demonstration is not straightforward. We need to adjust the initial data $\xi_i$, $\eta_i$ and $\zeta_i$ ($i=1,2$) delicately, not only to validate the splicing but also to ensure that the spliced subsolutions/supersolutions achieve the desired conditions on $\partial D$ and at infinity.

To begin with, we introduce several lemmas needed in the proof.
\begin{lemma}[Comparison principle]\label{pf-thm:comparison}
Let $f$ be as in \eqref{eq:pro-eq} with \eqref{eq:increase-f}, \eqref{eq:boundary-f} and \eqref{eq:nu-f} holding. Let $\Omega$ be a domain in $\mathbb{R}^n$ and $g\in C^0(\overline{\Omega})$ with $\inf_{\Omega} g>0$. Suppose that $u\in\mathrm{USC}(\overline{\Omega})\cap B_{p}(\Omega)$\footnote{$B_{p}(\Omega)$ denotes the set of functions that are bounded in $\Omega$ intersected with any ball of $\mathbb{R}^n$.} and $v\in\mathrm{LSC}(\overline{\Omega})\cap B_{p}(\Omega)$ are viscosity subsolution and viscosity supersolution of \eqref{eq:pro-eq} respectively. If $u\leq v$ on $\partial\Omega$ (and additionally
$$\lim_{|x|\to\infty}(u-v)(x)=0$$
provided $\Omega$ is unbounded), then $u\leq v$ in $\Omega$.
\end{lemma}

\begin{proof}
See \cite[Theorem A.3 and Corollary A.6]{Jiang-Li-Li-2021-b}.
\end{proof}

Thanks to Lemma \ref{pf-thm:comparison}, Perron's method as in \cite{Ishii1992,Ishii1989} could be adapted to the following version for equation \eqref{eq:pro-eq}.
\begin{lemma}[Perron's method]\label{pf-thm:Perron}
Let $f$, $g$ and $\Omega$ be as in Lemma \ref{pf-thm:comparison}. Let $\varphi\in C^0(\partial\Omega)$. Suppose that there exist $\underline{u},\bar{u}\in C^0(\overline{\Omega})$ such that
$$f(\lambda(D^2\underline{u}))\geq g(x)\geq f(\lambda(D^2\bar{u}))\quad\text{in }\Omega$$
in the viscosity sense, $\underline{u}\leq\bar{u}$ in $\Omega$ and $\underline{u}=\varphi$ on $\partial \Omega$ (and additionally
$$\lim_{|x|\to\infty}(\underline{u}-\bar{u})(x)=0$$
provided $\Omega$ is unbounded). Then
\begin{align*}
u(x):=\sup\{&v(x)| v\in\mathrm{USC}(\Omega), f(\lambda(D^2v))\geq g(x)\text{ in }\Omega\text{ in the viscosity}\\
&\text{sense, }\underline{u}\leq v\leq\bar{u}\text{ in }\Omega, v=\varphi\text{ on }\partial\Omega\}
\end{align*}
is in $C^0(\overline{\Omega})$ and is a viscosity solution of the problem
\begin{equation*}
\begin{cases}
f(\lambda(D^2u))=g(x)&\text{in }\Omega,\\
u=\varphi & \text{on }\partial\Omega.
\end{cases}
\end{equation*}
\end{lemma}

\begin{proof}
See \cite[Theorem B.1]{Jiang-Li-Li-2021-b}.
\end{proof}

To process boundary behavior of the solution, we need the following existence result of barrier functions.
\begin{lemma}\label{pf-lem:w-xi}
Let $D$ be a bounded strictly convex domain of $\mathbb{R}^n$ ($n\geq3$) with $\partial D\in C^2$ and let $\varphi\in C^2(\partial D)$. Let $K>0$ and let $A$ be an invertible and symmetric matrix. There exists some constant $C$, depending only on $n,\|\varphi\|_{C^2(\partial D)},K$, the upper bound of $A$, the diameter and the convexity of $D$, and the $C^2$ norm of $\partial D$, such that for every $\xi\in\partial D$, there exists $\bar{x}(\xi)\in\mathbb{R}^n$ satisfying $|\bar{x}(\xi)|\leq C$ and $$\omega_\xi<\varphi\quad\text{on }\partial D\setminus\{\xi\},$$ where $$\omega_\xi(x)=\varphi(\xi)+\frac{K}{2}\left[(x-\bar{x}(\xi))^TA(x-\bar{x}(\xi))-(\xi-\bar{x}(\xi))^TA(\xi-\bar{x}(\xi))\right]$$
for $x\in\mathbb{R}^n$.
\end{lemma}

\begin{proof}
 See \cite[Lemma 5.1]{Caffarelli-Li-2003} or \cite[Lemma 3.1]{Bao-Li-Li-2014}.
\end{proof}

We are now in position to prove Theorem \ref{thm:main}.
\begin{proof}[Proof of Theorem \ref{thm:main}]
We first observe that, by an orthogonal transformation and by subtracting a linear function from $u$, it suffices to prove in the case that the matrix $A$ is of diagonal form \eqref{sub-eq:diag-A} and the vector $b$ is $0$; see for instance \cite[Lemma 3.3]{Li-Li-2018} for a specific demonstration. We next split the proof into three steps.

For convenience, denote
$$B_r=\{x\in\mathbb{R}^n:|x|<r\}\quad\text{and}\quad D_r=\{x\in\mathbb{R}^n:\frac{1}{2}x^TAx<r\}.$$
Without loss of generality, we assume $B_1\subset D\subset D_{s_0}$, where $s_0$ is as in \eqref{sub-eq:g} and \eqref{sup-eq:g-s0}.

\textbf{Step 1}. Construct a viscosity subsolution $\underline{u}$ of \eqref{eq:pro-eq} with $\underline{u}=\varphi$ on $\partial D$ and the asymptotics
\begin{equation}\label{pf-eq:quadratic}
\lim_{|x|\to\infty}\left|\underline{u}(x)-\left(\frac12x^TAx+c\right)\right|=0.
\end{equation}
 The idea is to find two viscosity subsolutions $\underline{w}$ and $u_{\xi_1(m),\xi_2(c)}$ of \eqref{eq:pro-eq}, which attain the boundary value and the asymptotics respectively, and then splice them together as in \eqref{pf-eq:def-sub} below.

Let $K>0$ be large enough such that the function $\omega_\xi$ given by Lemma \ref{pf-lem:w-xi} satisfies\footnote{This can be guaranteed by conditions \eqref{eq:increase-f} and \eqref{eq:nu-f}, referring to Remark \ref{rk:A}.}
$$f(\lambda(D^2\omega_\xi))=f(K\lambda(A))\geq\sup_{\mathbb{R}^n\setminus D}g.$$
That is, $\omega_\xi$ is a smooth subsolution of \eqref{eq:pro-eq}. For $x\in\mathbb{R}^n\setminus D$, let us set $$\underline{w}(x)=\max\{\omega_{\xi}(x):\xi\in\partial D\}.$$
Then $\underline{w}$ is a viscosity subsolution of \eqref{eq:pro-eq} by \cite[Lemma 4.2]{Ishii1992} and $\underline{w}=\varphi$ on $\partial D$ by Lemma \ref{pf-lem:w-xi}.

For $\xi_1\in\mathbb{R}$ and $\xi_2>w_0(s_0)$, recall from Corollary \ref{sub-pro:subsolution} that Proposition \ref{sub-pro:ode} gives a smooth subsolution $u_{\xi_1,\xi_2}$ of \eqref{eq:pro-eq} in $\mathbb{R}^n\setminus\overline{D_{s_0}}$ which is defined by \eqref{sub-eq:ode-u-xi} and satisfies
$$\lim_{|x|\to\infty}\left|u_{\xi_1,\xi_2}(x)-\left(\frac12x^TAx+\xi_1+\mu(s_0,\xi_2)\right)\right|=0.$$
Here $\mu(s_0,\xi_2)$ is given by \eqref{sub-eq:def-mu}.
We claim there exists a constant $c_*$ such that for each $c>c_*$ one can choose proper $\xi_1$ and $\xi_2$ to fulfill
\begin{equation}\label{pf-eq:xi-c}
\xi_1+\mu(s_0,\xi_2)=c,
\end{equation}
which implies that $u_{\xi_1,\xi_2}$ attains \eqref{pf-eq:quadratic} this moment.
Simultaneously, in order to splice such $u_{\xi_1,\xi_2}$ with $\underline{w}$, we also require the choice of $\xi_1$ and $\xi_2$ to fulfill
\begin{equation}\label{pf-eq:splice-sub}
\max_{\partial D_{s_1}}u_{\xi_1,\xi_2}\leq\min_{\partial D_{s_1}}\underline{w}\quad\text{and}\quad\min_{\partial D_{s_2}}u_{\xi_1,\xi_2}\geq\max_{\partial D_{s_2}}\underline{w},
\end{equation}
where $s_1$ and $s_2$ are two fixed numbers such that $s_2>s_1>s_0$.

Indeed, recall from \eqref{sub-eq:ode-u-xi} that
$$u_{\xi_1,\xi_2}(x)=\xi_1+\int_{s_0}^{\frac12x^TAx}w_{\xi_2}(t)\,dt,\quad x\in\mathbb{R}^n\setminus D_{s_0},$$
where $w_{\xi_2}$ is given by Lemma \ref{sub-lem:PL-w}. Let $$m=\min\{\omega_{\xi}(x): \xi\in\partial D, x\in \overline{D_{s_1}}\setminus D\}.$$ Fixing $\xi_1=m-\int_{s_0}^{s_1}w_{\xi_2}(t)\,dt=:\xi_1(m)$ yields $$u_{\xi_1(m),\xi_2}(x)=m+\int_{s_1}^{\frac12x^TAx}w_{\xi_2}(t)\,dt.$$
Since $w_{\xi_2}>0$ for $x\in\mathbb{R}^n\setminus D_{s_0}$, $$u_{\xi_1(m),\xi_2}(x)\leq m\leq\min_{\partial D_{s_1}}\underline{w}$$
whenever $x\in\overline{D_{s_1}}\setminus D_{s_0}$ and $\xi_2>w_0(s_0)$. Namely, $u_{\xi_1(m),\xi_2}$ satisfies the first condition of \eqref{pf-eq:splice-sub}. Regarding the second one, it suffices to let $\xi_2$ be sufficiently large since $w_{\xi_2}$ is strictly increasing with respect to $\xi_2$ (see Lemma \ref{sub-lem:PL-w}). Thus, let us assume that $u_{\xi_1(m),\xi_2}$ satisfies \eqref{pf-eq:splice-sub} when $\xi_2>\bar{C}\geq w_0(s_0)$.

Now, by recalling \eqref{sub-eq:def-mu} we notice that
$$\xi_1(m)+\mu(s_0,\xi_2)=m+\mu(s_1,\xi_2).$$
Via Lemma \ref{sub-lem:PL-w}, we infer that $\mu(s_1,\xi_2)$ is increasing with respect to $\xi_2$ and satisfies $\lim_{\xi_2\to+\infty}\mu(s_1,\xi_2)=+\infty$.
Thus, let $c_*$ be a constant such that
\begin{equation}\label{pf-eq:c-star-1}
c_*\geq m+\mu(s_1,\bar{C}).
\end{equation}
Then for each $c>c_*$, there exists a unique $\xi_2(c)>\bar{C}$ such that
$$\xi_1(m)+\mu(s_0,\xi_2(c))=c,$$
illustrating that $u_{\xi_1(m),\xi_2(c)}$ achieves \eqref{pf-eq:xi-c}.

For $c>c_*$, we define
\begin{equation}\label{pf-eq:def-sub}
\underline{u}(x)=
\begin{cases}
\underline{w}(x),\quad &x\in D_{s_1}\setminus D,\\
\max\{\underline{w}(x),u_{\xi_1(m),\xi_2(c)}(x)\},\quad &x\in D_{s_2}\setminus D_{s_1},\\
u_{\xi_1(m),\xi_2(c)}(x),&x\in\mathbb{R}^n\setminus D_{s_2}.
\end{cases}
\end{equation}
From Definition \ref{def:visc} and \cite[Lemma 4.2]{Ishii1992}, we deduce that $\underline{u}$ is a viscosity subsolution of \eqref{eq:pro-eq} satisfying \eqref{pf-eq:quadratic} and $\underline{u}=\underline{w}=\varphi$ on $\partial D$.

\textbf{Step 2}. Construct a viscosity supersolution $\bar{u}$ of \eqref{eq:pro-eq} to satisfy
\begin{equation*}
\underline{u}\leq\bar{u}\text{ in }\mathbb{R}^n\setminus D\quad\text{and}\quad\lim_{|x|\to\infty}(\bar{u}-\underline{u})(x)=0.
\end{equation*}

If $\inf_{\mathbb{R}^n\setminus D}g=1$, then we can take $\bar{u}=\frac12x^TAx+c$ directly as the desired supersolution, provided that $c>c^*$ and $c^*$ is selected suitably; see for instance the argument in \cite{Jiang-Li-Li-2021} (or \cite{Bao-Li-Li-2014,Li-Li-2018,Li-Bao-2014,Li2019}).

While $\inf_{\mathbb{R}^n\setminus D}g<1$, the construction of $\bar{u}$ would not be straightforward since $\frac12x^TAx+c$ fails to be a supersolution of \eqref{eq:pro-eq} in this case. We will deal with that in a way analogous to the construction of $\underline{u}$. More precisely, we are going to find two supersolutions of \eqref{eq:pro-eq}, one of which coincides with $\underline{u}$ at infinity and the other of which surpasses $\underline{u}$ on $\partial D$, and then splice them together as in \eqref{pf-eq:def-super} below.

Let $\delta>0$ be such that $\alpha_\delta>1$ (see \eqref{sup-eq:def-alpha-d}) and let $\eta_2\in (0,W_0(s_0))$ (see \eqref{sup-eq:f-W0-g}). For $\eta_1\in\mathbb{R}$, we first recall from Corollary \ref{sup-pro:supersolution} that Proposition \ref{sup-pro:ode} gives a smooth supersolution $U_{\eta_1,\eta_2,\delta}$ of \eqref{eq:pro-eq} in $\mathbb{R}^n\setminus\overline{D_{\hat{s}}}$ which is defined by \eqref{sup-eq:def-U-eta} and satisfies
\begin{equation}\label{pf-eq:U-asym}
\lim_{|x|\to\infty}\left|U_{\eta_1,\eta_2,\delta}(x)-\left(\frac12x^TAx+\eta_1+\bar\mu\right)\right|=0.
\end{equation}
Here $\hat{s}:=\max\{\bar{s},s_0\}$ and $\bar\mu$ are both independent of $\eta_1$, which implies they are already fixed.

For $\zeta_1\in\mathbb{R}$ and $\zeta_2>1$, we proceed by recalling that Proposition \ref{sup-pro:rad-ode} establishes another smooth supersolution $v_{\zeta_1,\zeta_2}$ of \eqref{eq:pro-eq} in $\mathbb{R}^n\setminus\overline{B_{1}}$ which is defined by \eqref{sup-eq:def-v-zeta}, that is, $$v_{\zeta_1,\zeta_2}(x)=\zeta_1+\int_{\frac12\tilde{a}}^{\frac12\tilde{a}|x|^2}\bar{w}_{\zeta_2}(t)\,dt,$$
where $\bar{w}_{\zeta_2}$ is strictly increasing with respect to $\zeta_2$ and satisfies \eqref{sup-eq:d-w-zeta} and \eqref{sup-eq:w-zeta-infty}.

To obtain a supersolution of \eqref{eq:pro-eq} in $\mathbb{R}^n\setminus\overline{D}$, we splice $U_{\eta_1,\eta_2,\delta}$ and $v_{\zeta_1,\zeta_2}$ together by choosing suitable $\zeta_1$ and $\zeta_2$ such that
\begin{equation}\label{pf-eq:splice-sup}
\max_{\partial B_{r_1}}v_{\zeta_1,\zeta_2}\leq\min_{\partial B_{r_1}}U_{\eta_1,\eta_2,\delta}\quad\text{and}\quad\min_{\partial B_{r_2}}v_{\zeta_1,\zeta_2}\geq\max_{\partial B_{r_2}}U_{\eta_1,\eta_2,\delta},
\end{equation}
where $r_1$ and $r_2$ are two fixed numbers such that $D_{\hat{s}}\subset B_{r_1}\subset B_{r_2}$. Indeed, let
\begin{equation*}
M(\eta_1)=\min\{U_{\eta_1,\eta_2,\delta}(x):x\in\overline{B_{r_1}}\setminus D_{\hat{s}}\}
\end{equation*}
and fix $$\zeta_1=M(\eta_1)-\int_{\frac12\tilde{a}}^{\frac12\tilde{a}r_1^2}\bar{w}_{\zeta_2}(t)\,dt=:\bar{\zeta_1}.$$ Then
\begin{equation}\label{pf-eq:sup-v-zeta}
v_{\bar{\zeta_1},\zeta_2}(x)=M(\eta_1)+\int_{\frac12\tilde{a}r_1^2}^{\frac12\tilde{a}|x|^2}\bar{w}_{\zeta_2}(t)\,dt.
\end{equation}
Clearly, $v_{\bar{\zeta_1},\zeta_2}$ satisfies the first condition in \eqref{pf-eq:splice-sup} whenever $\zeta_2>1$. To tackle the second one, by recalling \eqref{sup-eq:def-U-eta} we observe that
\begin{equation*}
\max_{\partial B_{r_2}}U_{\eta_1,\eta_2,\delta}=\eta_1+\max_{\partial B_{r_2}}\int_{s_0}^{\frac12x^TAx}W_{\eta_2,\delta}(t)\,dt
\end{equation*}
 and
\begin{align*}
\min_{\partial B_{r_2}}v_{\bar{\zeta_1},\zeta_2}
=&M(\eta_1)+\min_{\partial B_{r_2}}\int_{\frac12\tilde{a}r_1^2}^{\frac12\tilde{a}|x|^2}\bar{w}_{\zeta_2}(t)\,dt\\
=&\eta_1+\min\left\{\int_{s_0}^{\frac12x^TAx}W_{\eta_2,\delta}(t)\,dt:x\in\overline{B_{r_1}}\setminus D_{\hat{s}}\right\}+\int_{\frac12\tilde{a}r_1^2}^{\frac12\tilde{a}r_2^2}\bar{w}_{\zeta_2}(t)\,dt.
\end{align*}
Hence, by the monotonicity of $\bar{w}_{\zeta_2}$ with respect to $\zeta_2$, one can infer that if $\zeta_2$ is sufficiently large (fix $\zeta_2=\bar{\zeta_2}$), then
$v_{\bar{\zeta_1},\bar{\zeta_2}}$ satisfies \eqref{pf-eq:splice-sup} whatever $\eta_1$ is.

Now, in addition to \eqref{pf-eq:c-star-1}, let the constant $c_*$ further satisfy
\begin{equation}\label{pf-eq:c-star-2}
c_*\geq \max_{\partial D}\varphi-\int_{\frac12\tilde{a}r_1^2}^{\frac12\tilde{a}}\bar{w}_{\bar{\zeta}_2}(t)\,dt+\bar{\mu}.
\end{equation}
For $c>c_*$, we set
\begin{equation}\label{pf-eq:eta-1-c}
\eta_1(c):=c-\bar{\mu}
\end{equation}
and define
\begin{equation}\label{pf-eq:def-super}
\bar{u}(x)=
\begin{cases}
v_{\bar{\zeta_1},\bar{\zeta_2}}(x),&x\in B_{r_1}\setminus D,\\
\min\{v_{\bar{\zeta_1},\bar{\zeta_2}}(x),U_{\eta_1(c),\eta_2,\delta}(x)\},&x\in B_{r_2}\setminus B_{r_1},\\
U_{\eta_1(c),\eta_2,\delta}(x),&x\in\mathbb{R}^n\setminus B_{r_2}.
\end{cases}
\end{equation}
It is seen that $\bar{u}$ is a viscosity supersolution of \eqref{eq:pro-eq}, and by  \eqref{pf-eq:U-asym} it holds that
\begin{equation}\label{pf-eq:sup-infinity}
\lim_{|x|\to\infty}\left|\bar{u}(x)-\left(\frac12x^TAx+c\right)\right|=\lim_{|x|\to\infty}\left|U_{\eta_1(c),\eta_2,\delta}(x)-\left(\frac12x^TAx+c\right)\right|=0.
\end{equation}
Moreover, in view of \eqref{pf-eq:sup-v-zeta} and \eqref{sup-eq:def-U-eta}, we find that
\begin{equation*}
\bar{u}=v_{\bar{\zeta_1},\bar{\zeta_2}}\geq M(\eta_1(c))+\int_{\frac12\tilde{a}r_1^2}^{\frac12\tilde{a}}\bar{w}_{\bar{\zeta}_2}(t)\,dt\geq\eta_1(c)+\int_{\frac12\tilde{a}r_1^2}^{\frac12\tilde{a}}\bar{w}_{\bar{\zeta}_2}(t)\,dt\quad\text{on }\partial D.
\end{equation*}
Using \eqref{pf-eq:c-star-2} and \eqref{pf-eq:eta-1-c} gives
\begin{equation}\label{pf-eq:sup-boundary}
\bar{u}(x)\geq c_*-\bar{\mu}+\int_{\frac12\tilde{a}r_1^2}^{\frac12\tilde{a}}\bar{w}_{\bar{\zeta}_2}(t)\,dt\geq\max_{\partial D}\varphi,\quad x\in\partial D.
\end{equation}
In view of \eqref{pf-eq:sup-infinity} and \eqref{pf-eq:sup-boundary}, we apply Lemma \ref{pf-thm:comparison} to obtain
$$\underline{u}\leq\bar{u}\quad\text{in }\mathbb{R}^n\setminus D.$$

\textbf{Step 3}. Construct a viscosity solution $u$ to problem \eqref{eq:pro}.

With $\underline{u}$ and $\bar{u}$ above, we define
\begin{align*}
u(x)=&\sup\{v(x)|v\in\mathrm{USC}(\mathbb{R}^n\setminus\overline{D}), f(\lambda(D^2v))\geq g\text{ in }\mathbb{R}^n\setminus\overline{D}\text{ in the viscosity}\\
&\text{sense, }\underline{u}\leq v\leq\bar{u}\text{ in }\mathbb{R}^n\setminus\overline{D}\text{ and }v=\varphi\text{ on }\partial D\}.
\end{align*}
Thanks to \eqref{pf-eq:quadratic} and \eqref{pf-eq:sup-infinity},
$$\lim_{|x|\to\infty}\left|u(x)-\left(\frac12x^TAx+c\right)\right|=0.$$
Consequently, by Lemma \ref{pf-thm:Perron}, we conclude that $u\in C^0(\mathbb{R}^n\setminus D)$ is a viscosity solution to problem \eqref{eq:pro}.

Finally, the uniqueness of $u$ follows from Lemma \ref{pf-thm:comparison}. This completes the proof.
\end{proof}

\begin{remark}\label{pf-rk:proof}
One may note from the proof that the approach to constructing a desired viscosity supersolution of \eqref{eq:pro-eq} in \textbf{Step 2} is more subtle than that in \textbf{Step 1} for the desired subsolution. Specifically, the latter is to choose two proper parameters $\xi_1$ and $\xi_2$ to splice subsolutions $u_{\xi_1,\xi_2}$ and $\underline{w}$ as one viscosity subsolution with prescribed boundary data and asymptotic behavior at infinity, which is analogous to those presented in other related works, for instance, \cite{Bao-Li-Li-2014,Bao-Li-Zhang-2015,Caffarelli-Li-2003,Cao-Bao-2017,Jiang-Li-Li-2021,Jiang-Li-Li-2021-b,Li-Li-2018,Li2019}. By contrast, in the former we had to carefully adjust three parameters $\eta_1$, $\zeta_1$ and $\zeta_2$ from two different supersolutions $U_{\eta_1,\eta_2,\delta}$ (where $\eta_2$ and $\delta$ are preseted) and $v_{\zeta_1,\zeta_2}$, in order to splice them validly as one viscosity supersolution achieving the expected condition on $\partial D$ and that at infinity simultaneously. This differs from the above literatures where the supersolution is obtained in a direct way without splicing instead.

Indeed, when $f$ is of a special form \eqref{eq:k-sigma} or \eqref{eq:quotient}, thanks to an explicit formula for $\sigma_k$ acting on generalized symmetric functions \cite{Bao-Li-Li-2014}, the authors of \cite{Bao-Li-Zhang-2015,Cao-Bao-2017,Jiang-Li-Li-2021-b} could parallel obtain a family of explicit subsolutions and supersolutions in $\mathbb{R}^n\setminus\overline{D}$. However, when $f$ is of an abstract form, our seeking $u_{\xi_1,\xi_2}$ in Section \ref{sec:sub} and $U_{\eta_1,\eta_2,\delta}$ in Subsection \ref{sec:sup-1} are not strictly parallel since in the latter we introduced an extra parameter $\delta$ (compare Lemma \ref{sup-lem:upper-f} with Lemma \ref{sub-lem:lower-f}). More importantly, unlike $u_{\xi_1,\xi_2}$, the domain where $U_{\eta_1,\eta_2,\delta}$ becomes a supersolution varies with parameters $\eta_2$ and $\delta$ involved (see Remark \ref{sup-rk:bar-s}), implying it generally does not work near $\partial D$. Therefore, we were naturally led to find another supersolution $v_{\zeta_1,\zeta_2}$ with fine properties so that it would achieve the expected condition on $\partial D$ and also could validate the process of splicing $U_{\eta_1,\eta_2,\delta}$ somewhere near $\partial D$, as already described above. To the best of our knowledge, this is a new ingredient among the proofs of such exterior problems presented in the literatures mentioned above.
\end{remark}

\section*{Acknowledgements}
The authors would like to thank professor Jiguang Bao for helpful comments during the preparation of this work. The first author was supported by China Scholarship Council. Part of this work has been done while the first author was visiting the Department of Mathematics ``Federigo Enriques'' of Universit\`a degli Studi di Milano, which is acknowledged for the hospitality. The authors would also like to thank the referee for the careful reading and valuable comments on the original manuscript.

\end{document}